\documentclass[a4paper,11pt]{article}
\usepackage{latexsym}
\usepackage{amssymb}
\usepackage{enumerate}
\usepackage{enumitem}   
\usepackage{xcolor}
\usepackage{commath}
\usepackage{amsfonts}
\usepackage{amsmath,amsthm}
\usepackage{hyperref}
\usepackage{algorithm}
\usepackage{algorithmic}
\usepackage{framed}
\usepackage{boites}
\usepackage{graphicx} % otherwise
\usepackage{comment}

\newenvironment{@abssec}[1]{%
    \if@twocolumn

      \section*{#1}%
    \else

      \vspace{.05in}\footnotesize
      \parindent .2in
 {\upshape\bfseries #1. }\ignorespaces
    \fi}

    {\if@twocolumn\else\par\vspace{.1in}\fi}

\newenvironment{keywords}{\begin{@abssec}{\keywordsname}}{\end{@abssec}}

\newenvironment{AMS}{\begin{@abssec}{\AMSname}}{\end{@abssec}}

\newcommand\keywordsname{Key words}
\newcommand\AMSname{AMS subject classifications}
\newcommand\AMname{AMS subject classification}
\newcommand\restr[2]{{% we make the whole thing an ordinary symbol
\left.\kern-\nulldelimiterspace % automatically resize the bar with \right
#1 % the function
\vphantom{|} % pretend it's a little taller at normal size
\right|_{#2} % this is the delimiter
}}
\newtheorem{theorem}{Theorem}[section]
\newtheorem{lemma}[theorem]{Lemma}
\newtheorem{corollary}[theorem]{Corollary}

\newtheorem{definition}[theorem]{Definition}

\newcommand{\measurerestr}{%
  \,\raisebox{-.127ex}{\reflectbox{\rotatebox[origin=br]{-90}{$\lnot$}}}\,%
}

\newcommand{\RR}{\mathbb{R}}

\def\XXint#1#2#3{{\setbox0=\hbox{$#1{#2#3}{\int}$}
\vcenter{\hbox{$#2#3$}}\kern-.5\wd0}}

\newcommand{\link}{\mathop{\circ\kern-.35em -}}

\newcommand{\ol}{\overline}
\newcommand{\pa}{\partial}

\newcommand{\dv}{\mathop{\mathrm{div}}}

\newcommand{\gr}{\nabla}

\newcommand{\al}{\alpha}
\newcommand{\be}{\beta}
\newcommand{\ga}{\gamma}  
\newcommand{\Ga}{\Gamma}

\newcommand{\Si}{\Sigma}
\newcommand{\te}{\theta}

\newcommand{\om}{\omega}
\newcommand{\Om}{\Omega}
\newcommand{\rn}{{\mathbb{R}}^N}

\newcommand{\cF}{\mathcal{F}}
\newcommand{\cG}{{\mathcal G}}
\newcommand{\cH}{{\mathcal H}}

\newcommand{\cL}{{\mathcal L}}

\newcommand{\cX}{\mathcal{X}}

\title{The double queen Dido's problem \thanks{This research was partially supported by the Grants-in-Aid for Scientific Research (B)  No.18H01126 and  JSPS Fellows No.18J11430 of Japan Society for the Promotion of Science. This work was also partially supported by the project ANR-18-CE40-0013 SHAPO financed by the French Agence Nationale de la Recherche (ANR).}}

\author{Lorenzo Cavallina\thanks{Research Center for Pure and Applied Mathematics,
Graduate School of Information Sciences, Tohoku
University, Sendai, 980-8579, Japan ({\tt cava@ims.is.tohoku.ac.jp}, {\tt sigersak@tohoku.ac.jp}).}\ ,
Antoine Henrot\thanks{Institut Elie Cartan de Lorraine, CNRS UMR 7502 and Universit\'e de Lorraine, BP 70239 54506 Vandoeuvre-l\`es-Nancy, France
({\tt antoine.henrot@univ-lorraine.fr}).}\ , and Shigeru Sakaguchi\footnotemark[2]}

\date{}

\begin{document}

\maketitle

\begin{abstract}
%(from the proposed abstract of my talk at Lorraine University, Nancy)
%It is well known that the ball has the least surface area among all sets with given volume. Moreover, under volume constraint, half-balls minimize the relative perimeter inside half-spaces: this constrained minimization problem in two dimensions is also known as queen Dido's problem and is one of the oldest problems of calculus of variations ever studied.     
 This paper deals with a variation of the classical isoperimetric problem in dimension $N\ge 2$ for a two-phase piecewise constant density whose discontinuity interface is a given hyperplane. We introduce a weighted perimeter functional with three different weights, one for the hyperplane and one for each of the two open half-spaces in which $\mathbb{R}^N$ gets partitioned.  
We then consider the problem of characterizing the sets $\Omega$ that minimize this weighted perimeter functional 
under the additional constraint that the volumes of the portions of $\Omega$ in the two half-spaces are given.  
It is shown that the problem admits two kinds of minimizers, which will be called type~I and type~II,  respectively. These minimizers are made of the union of two spherical domes whose angle of incidence satisfies some kind of \textquotedblleft Snell's law\textquotedblright.  Finally, 
we  provide a complete classification of the minimizers depending on the various parameters of the problem.
\end{abstract}



\pagestyle{plain}
\thispagestyle{plain}

\section{Introduction and main result}

%\textcolor{red}{
The classical isoperimetric inequality has a long history even if one had to wait until the fifties for a rigorous proof for general sets by E. De Giorgi (see \cite{Degiorgi, Burago Zalgaller} for
some history). There is also some important literature on isoperimetric problems with densities, see for 
example \cite[Chapter 18]{Morgan 2008} for an introduction. Let a positive function 
$f:\mathbb{R}^N\to \mathbb{R}_+$ be given.
For any sufficiently smooth set $E$, we define the weighted volume and perimeter
of E to be 
\begin{equation}\label{weighted area and perimeter}
|E|_f=\int_E f\,dx, \quad P_f(E)=\int_{\partial E} f \,d\cH^{N-1}.
\end{equation}
In probability theory, it is quite common to use the Gaussian density $f(x)=\exp(-|x|^2)$
for which the isoperimetric sets are half-spaces, see \cite{borell} and \cite[Chapter 18]{Morgan 2008}.
Another classical choice is radial functions like $f(x)=|x|^q$ (see e.g. \cite{ABCMP, BBMP, CJQW}), for which the isoperimetric sets are usually balls. A much less studied density is a 
 piecewise constant density. In the paper \cite{CMV2010}, A. Ca{\~n}ete, M. Miranda and N. Vittone studied some particular cases related to the characteristic functions of half-planes, strips and balls. Our aim here
 is to consider a variant of this study by considering two-half spaces in $\rn$ with different constant
 densities together with a cost $\ga\ge0$ (possibly $0$) on the hyperplane which is the interface.
%}
%\textcolor{red}{
Concerning Dido's problem in the half space (with a constant density), 
the solution is given by a half ball, see e.g.
\cite{CGR2007} for the proof of a more general result, namely that the half ball has the smallest possible (relative) perimeter than any other set of the same volume outside a convex domain. As an application, for our problem, if we do not put any cost on the interface (that is, $\ga=0$), the problem decouples and the solution of the isoperimetric problem will be the union of two half balls. %This is the case where $\gamma = 0$ for the nonnegative parameter $\gamma$ that will be introduced below. 
%}

\medskip
%\textcolor{red}{
Let us now fix the notations and set the problem in more detail.
%}
Let $\rn_\pm$ denote the following left and right open half-spaces of $\rn$:
\begin{equation*}
\rn_-=\left\{ (x_1,x_2,\dots, x_N)\in \rn \;:\; x_1<0 \right\},\;
\rn_+=\left\{ (x_1,x_2,\dots, x_N)\in \rn \;:\; x_1>0 \right\},
\end{equation*}
and let  $\Sigma$ be the vertical hyperplane 
$$
\Sigma=\left\{ (x_1,x_2,\dots, x_N)\in \rn \;:\; x_1=0 \right\}.
$$
For a given set of finite perimeter $\Omega\subset\rn$, put
\begin{equation}\label{some notations}
\Om_\pm = \Om\cap \rn_\pm,\quad \Ga_\pm= \pa\Om\cap \rn_\pm, \quad \Ga_0= \pa\Om\cap\Sigma.    
\end{equation}

%For given constants $V_\pm,\rho_\pm>0$ and $\gamma \ge 0$, we consider the following constrained minimization problem:
For given constants $V_\pm,\rho_\pm>0$ and $\gamma \ge 0$, we consider the following constrained minimization problem:
\begin{equation}\label{min pb}
    \min\left\{\cF(\Om)\;:\; \Om\subset \rn \text{ is of finite perimeter and } \rho_\pm |\Om_\pm|=V_\pm\right\},
\end{equation}
where
\begin{equation}
\cF(\Om)= \rho_- P(\Omega, \rn_-)+ \rho_+ P(\Omega, \rn_+) + \gamma \cH^{N-1}(\Ga_0).
%\cF(\Om)= \rho_- |\Ga_-|+ \rho_+ |\Ga_+| + \gamma |\Ga_0|,    
\end{equation}
In the above, we used the notations $|\cdot|$, $P$ and $\cH^{N-1}$ for the Lebesgue measure, the relative perimeter
in the sense of De Giorgi (see \cite{HP}) and the $(N-1)$-dimensional Hausdorff measure,  respectively. We remark that, if the boundary $\pa\Om$ coincides with the reduced boundary $\pa^\ast \Om$, then $$\cF(\Om)=P_f(\Om),$$ where $P_f$ is the weighted perimeter introduced in \eqref{weighted area and perimeter} and $f$ is the piecewise constant function defined as  $f(x)=\rho_\pm$ for $x_1\gtrless 0$ and $f(x)=\ga$ if $x_1=0$. 
%\textcolor{red}{is it true?}
%In what follows, when no confusion is possible, we will usually omit the subscript $\ga$ and just write $\cF$ instead. 

The aim of this paper is to give a complete characterization of the minimizers of \eqref{min pb} for all values of the parameters $V_\pm$, $\rho_\pm$ and $\gamma$.  
This paper is organized as follows. 
In section \ref{ch geometr prop} we show that, if $\ga>0$, any minimizer $\Om$ of \eqref{min pb}, if it exists, must be connected and both $\Ga_\pm$ are spherical caps. This fact allows for only two types of minimizers of \eqref{min pb}: one where the boundaries of the two spherical caps coincide (type~I) and another one where $\cH^{N-1}(\Ga_0)>0$ (type~II).
In section \ref{ch existence} we show the existence of minimizers for problem \eqref{min pb} by means of a standard compactness argument. %We also prove \textcolor{red}{regularity of the minimizers}. 
 In section \ref{ch snell law} we derive some geometrical transmission conditions that describe the angle of incidence between $\Ga_\pm$ and $\Sigma$. By means of these conditions, we are able to reduce the number of potential minimizers to just two: one for each type (up to translations). Finally, in section \ref{ch proof main thm} we find a threshold $\ga^*=\ga^*(V_\pm,\rho_\pm)$ such that the minimizer of \eqref{min pb} is of type~II for $0<\ga<\ga^*$ and of type~I for $\ga\ge\ga^*$. 

%%%%%%%%%%%%%%%%%%%%%%%%%%%%%%%%%%%%%%%%%%%%%%%%%
%%%%%%%%%%%%%%%%%%%%%%%%%%%%%%%%%%%%%%%%%%%%%%%%%
%%%%%%%%%%%%%%%%%%%%%%%%%%%%%%%%%%%%%%%%%%%%%%%%%
%%%%%%%%%%%%%%%%%  Section 2 begins  %%%%%%%%%%%%%%%%%%%%%%%
%%%%%%%%%%%%%%%%%%%%%%%%%%%%%%%%%%%%%%%%%%%%%%%%%
%%%%%%%%%%%%%%%%%%%%%%%%%%%%%%%%%%%%%%%%%%%%%%%%%
%%%%%%%%%%%%%%%%%%%%%%%%%%%%%%%%%%%%%%%%%%%%%%%%%

\section{Geometrical properties of minimizers} \label{ch geometr prop}
%\begin{itemize}
 %   \item \textcolor{red}{any minimizer is a critical shape for the Lagrangian $\implies$ mean curvature $H$ is constant in each half-space $\rn_\pm$.}
%\item\textcolor{red}{
%Schwarz symmetrization in the direction $e_1$ on each connected component of $\Om_\pm$ $\implies$ $\Ga_\pm$ are disjoint unions of spheres and/or spherical caps with the same radius $R_\pm$}
%\item\textcolor{red}{
%``Gluing together different components decreases the surface area"$\implies$ any minimizer is connected.
%}
%\end{itemize}
Here we will study the geometrical properties of a minimizer of \eqref{min pb} provided that at least one exists. The question of existence will be then addressed in the next section. For this purpose, let us first utilize the Schwarz symmetrization (see \cite[p. 238]{Maggi} for its definition and properties). Let $\ell$ be a line orthogonal to $\Sigma$. For a set of finite perimeter $\Omega$ in $\mathbb R^N$ with $|\Omega| < +\infty$, let $\Omega^*$ denote the Schwarz symmetrization of $\Omega$ around $\ell$. Then we have

%%%%%%%%%%%%%%%%%%%%%%%%%%%%%%%%%%%%%%%%%%%%%%%%%
%%%%%%%%%%%%%%%%% Lemma 2.1 begins %%%%%%%%%%%%%%%%%%%%%%%
%%%%%%%%%%%%%%%%%%%%%%%%%%%%%%%%%%%%%%%%%%%%%%%%%

\begin{lemma}
\label{Schwarz symmetrization A}
If $\Omega$ is a set of finite perimeter in $\mathbb R^N$ with $|\Omega| < +\infty$, then $\Omega^*$ is also a set of finite perimeter in $\mathbb R^N$ and the following hold
$$
|\Om^*_\pm| = |\Om_\pm|,\ P(\Omega^*, \rn_\pm) \le P(\Omega, \rn_\pm),\ \cH^{N-1}(\Ga^*_0) \le \cH^{N-1}(\Ga_0) \ \mbox{ and hence }\ \cF(\Om^*) \le \cF(\Om),
$$
respectively, where $\Om^*_\pm$ and $\Ga^*_0$ follow notations \eqref{some notations}.
\end{lemma}
\begin{proof} %Let us consider the former half. 
By means of the Schwarz symmetrization, Fubini's theorem yields that $|\Om^*_\pm| = |\Om_\pm|$, respectively.
Since the regularity of $\partial\Omega$ on $\Sigma$ is not transparent, we employ the following approximation argument.
In view of \cite[Exercise 15.13, p. 173]{Maggi}, we may find a decreasing sequence of positive numbers $\{ \varepsilon_n \}_n$ with $\lim\limits_{n \to \infty}\varepsilon_n=0$ such that for every $n \in \mathbb N$
each left half-space $H_{-\varepsilon_n} = \{ x \in \mathbb R^N : x_1< -\varepsilon_n \}$ satisfies the following:
\begin{eqnarray}
&&\mu_{\Omega\cap H_{-\varepsilon_n}} = \mu_\Omega\measurerestr H_{-\varepsilon_n} + e_1 \cH^{N-1}\measurerestr(\Omega\cap\partial H_{-\varepsilon_n}), \label{the key identity} \\
&&\cH^{N-1}(\Omega\cap\partial H_{-\varepsilon_n}) \le P(\Omega, H_{-\varepsilon_n}) \ \mbox{ and } P(\Omega\cap H_{-\varepsilon_n}) \le P(\Omega),
\nonumber
\end{eqnarray}
where $\mu_{\Omega\cap H_{-\varepsilon_n}}$ and $\mu_\Omega$ are the Gauss-Green measures of the two sets of finite perimeter $\Omega\cap H_{-\varepsilon_n}$ and $\Omega$, respectively, (see \cite[Chapter 12]{Maggi} for the definition and some basic properties) and $e_1=(1,0, \dots,0) \in \mathbb R^N$. Hence it follows from \eqref{the key identity} that for every $n \in \mathbb N$
\begin{equation}
\label{identity 1 for n}
P(\Omega\cap H_{-\varepsilon_n}) = P(\Omega, H_{-\varepsilon_n}) + \cH^{N-1}(\Omega\cap\partial H_{-\varepsilon_n}).
\end{equation}
By the same argument as above, we may also have that for every $n \in \mathbb N$
\begin{eqnarray}
&&P(\Omega^*\cap H_{-\varepsilon_n}) = P(\Omega^*, H_{-\varepsilon_n}) + \cH^{N-1}(\Omega^*\cap\partial H_{-\varepsilon_n}),
\label{identity 2 for n star}
\\
&&\cH^{N-1}(\Omega^*\cap\partial H_{-\varepsilon_n}) \le P(\Omega^*, H_{-\varepsilon_n}) \ \mbox{ and } P(\Omega^*\cap H_{-\varepsilon_n}) \le P(\Omega^*), \label{estimates from above for every n}
\end{eqnarray}
if we replace $\Omega$ with its Schwarz symmetrization $\Omega^*$.
By means of the Schwarz symmetrization, we notice that for every $n \in \mathbb N$
\begin{equation}
\label{means of the Schwarz symmetrization}
\left(\Omega\cap H_{-\varepsilon_n}\right)^* = \Omega^*\cap H_{-\varepsilon_n} \ \mbox{ and }\ \cH^{N-1}(\Omega\cap\partial H_{-\varepsilon_n}) =
\cH^{N-1}(\Omega^*\cap\partial H_{-\varepsilon_n}).
\end{equation}
Then, the inequality \cite[Theorem 19.11, p. 238]{Maggi} shows that
$$
P(\Omega^*\cap H_{-\varepsilon_n}) = P\left((\Omega\cap H_{-\varepsilon_n})^*\right) \le P(\Omega\cap H_{-\varepsilon_n}).
$$
Therefore, by combining \eqref{identity 1 for n} and \eqref{identity 2 for n star} with the second equality of \eqref{means of the Schwarz symmetrization}, we conclude that for every $n \in \mathbb N$
$$
P(\Omega^*, H_{-\varepsilon_n}) \le P(\Omega, H_{-\varepsilon_n}).
$$
Now, since the monotonically increasing sequence of sets $\{H_{-\varepsilon_n}\}_n$ converges to $\rn_-$, letting $n\to\infty$ gives
\begin{equation}
\label{left-half part}
P(\Omega^*, \rn_-) \le P(\Omega, \rn_-).
\end{equation}
Similarly, using right half-space $H_{\varepsilon_n} = \{ x \in \mathbb R^N : x_1 > \varepsilon_n \}$ instead of $H_{-\varepsilon_n}$ gives
\begin{equation}
\label{right-half part}
P(\Omega^*, \rn_+) \le P(\Omega, \rn_+).
\end{equation}
Also, we introduce the set $F_n = \{ x \in \mathbb R^N : -\varepsilon_n < x_1 < \varepsilon_n \}$ instead of the two half-spaces, and similarly we obtain that for every $n \in \mathbb N$
\begin{eqnarray*}
P(\Omega\cap F_n) &=& P(\Omega, F_n) + \cH^{N-1}(\Omega\cap\partial H_{-\varepsilon_n}) + \cH^{N-1}(\Omega\cap\partial H_{\varepsilon_n}),
\\
P(\Omega^*\cap F_n) &=& P(\Omega^*, F_n) + \cH^{N-1}(\Omega^*\cap\partial H_{-\varepsilon_n}) + \cH^{N-1}(\Omega^*\cap\partial H_{\varepsilon_n}).
\end{eqnarray*}
Here the inequality \cite[Theorem 19.11, p. 238]{Maggi} shows again that
$$
P(\Omega^*\cap F_n) \le P(\Omega\cap F_n),
$$
and hence similarly
$$
P(\Omega^*, F_n) \le P(\Omega, F_n).
$$
Now, since the monotonically decreasing sequence of sets $\{F_n\}_n$ converges to $\Si$, letting $n \to \infty$ gives
\begin{equation}
\label{middle part}
\cH^{N-1}(\Ga^*_0) \le \cH^{N-1}(\Ga_0).
\end{equation}
Finally, collecting \eqref{left-half part}, \eqref{right-half part} and \eqref{middle part} yields that
\begin{equation*}
\cF(\Om^*) \le \cF(\Om).
\end{equation*}

\end{proof}

%%%%%%%%%%%%%%%%%%%%%%%%%%%%%%%%%%%%%%%%%%%%%%%%%
%%%%%%%%%%%%%%%%% Theorem 2.2 begins %%%%%%%%%%%%%%%%%%%%%%%
%%%%%%%%%%%%%%%%%%%%%%%%%%%%%%%%%%%%%%%%%%%%%%%%%

%\begin{theorem}\label{thm spherical caps}
%If $\Om$ is a minimizer of \eqref{min pb}, then $\Ga_\pm$ are spherical caps.
%\end{theorem}
\begin{theorem}\label{thm spherical caps}
If $\Om$ is a minimizer of \eqref{min pb}, then each of $\Ga_\pm$ is either a spherical cap or a sphere.
\end{theorem}
\begin{proof}
Let $\Om$ be a minimizer of \eqref{min pb}.
%{\bf $\Ga_\pm$ have constant mean curvature.}
%
%First of all, notice that both $\Om_\pm$ are isoperimetric sets on their own, meaning that they minimize the quantities $\cH^{N-1}(\Ga_\pm)$ under the volume constraint $|\Om_\pm|=V_\pm/\rho_\pm$. In particular, by standard arguments involving shape derivatives computed with respect to smooth perturbations acting n $\Ga_-$ and $\Ga_+$ respectively, we obtain that $\Ga_-$ and $\Ga_+$ must have constant mean curvature (say $H_-$ and $H_+$ respectively). \textcolor{red}{say more?}
In particular, both $\Om_\pm$ are isoperimetric sets on their own, that is, $\Omega_\pm$ minimize perimeter in $\mathbb R^N_\pm$, respectively, with a volume constraint in the sense of Gonzales, Massari and Tamanini \cite[p. 27]{GMT1983}. Therefore their regularity result \cite[Theorem 2, p. 29]{GMT1983} implies that $\Ga_\pm$ must be analytic surfaces up to a singular set of dimension at most $N-8$ (that is to say that the singular set is empty when $N\le 7$).

It follows from Lemma \ref{Schwarz symmetrization A} that the Schwarz symmetrization $\Om^*$ of $\Omega$ is also a minimizer of \eqref{min pb} with $\cF(\Om^*) = \cF(\Om)$, since $\Om$ does.
Hence the equalities also hold in \eqref{left-half part} and \eqref{right-half part}.
These two equalities together with \cite[Theorem 19.11, p. 238]{Maggi} yield that for almost every $t \in \mathbb R$, the set $\Omega_t = \{ x \in \Omega : x_1 = t \}$ is $\cH^{N-1}$-equivalent to an $(N-1)$-dimensional ball, whose radius will be denoted by $R(t)\ge0$.
By combining this information with the regularity of $\Ga_\pm$ mentioned at the beginning of the proof, we infer that each $\Om_\pm$ needs to enjoy axial symmetry with respect to some straight line orthogonal to the hyperplane $\Sigma$ %(that, without loss of generality we can assume to be the $x_1$ axis) 
and each $\Ga_\pm$ does not contain any flat parts. In particular, both $\Ga_\pm$ are analytic everywhere except at most at those points where $\Ga_\pm$ intersect their axis of symmetry. In other words, we know that $\Ga_\pm$ are analytic at every point $x\in\Ga_\pm$ whose first component $x_1$ belongs to% the set
\begin{equation*}
I=\{t\in\RR \;:\; R(t)>0\}.
\end{equation*}

%\textcolor{red}{(maybe we need to add some reference to spherical symmetrization of sets? I think to remember that either Prof. Henrot or Prof. Sakaguchi knew some)}

%Except for the sets $\Om_\pm$ being already axially symmetric, there is only one other case where the symmetrized sets have the same surface area of the original $\Ga_-$ and $\Ga_+$, that is, only if they contain some flat parts. Nevertheless, this cannot happen because that would yield a singular set of dimension $N-2$, thus contradicting the regularity result of \cite{GMT1983}.

We will now show that each $\Ga_\pm$ is a spherical cap or a sphere, as claimed.
%$\Ga_\pm$ are spherical caps, as claimed.
%To fix ideas, let us consider the set $\Om_+$. %By the above, we know that $\Om_+$ is an axially symmetric set. 
%thus there exists a positive value $t>0$ such that the intersection $D=\ol{\Om_+}\cap \{x_1=t\}$ is an $(N-1)$-dimensional closed ball of positive radius. 
%Notice that if $D$ is a point, then $\Om_+$ must be a ball by the classical isoperimetric inequality.
%On the other hand, if $D$ is some ball of positive radius, let us consider the following rearrangement of $\Om_+$ as in the latter proof of Lemma \ref{Schwarz symmetrization}. Let $B$ denote the closed ball in $\rn$ determined by \begin{equation}\label{prop of B}
To fix ideas, let us consider the following rearrangement of $\Om_+$. Take some positive value $t_0\in I$. By the above, we know that $\Om_+$ is an axially symmetric set and the intersection $D=\ol{\Om_+}\cap \{x_1=t_0\}$ is an $(N-1)$-dimensional closed ball of positive radius $R(t_0)>0$. 
Now, let $B$ denote the closed ball in $\rn$ determined by \begin{equation}\label{prop of B}
B\cap \{x_1=t_0\}=D\ \mbox{ and }\ |B\cap \{x_1>t_0\}|=|\Om_+\cap\{x_1>t_0\}|.
\end{equation}
Define
\begin{equation}\label{a rearrangement}
\widetilde B = \left( B\cap \{x_1\le t_0\} \right) \cup \left(\overline{\Om_+}\cap \{x_1>t_0\}\right).
\end{equation}

\begin{figure}[ht]
\centering
\includegraphics[width=0.6\linewidth]{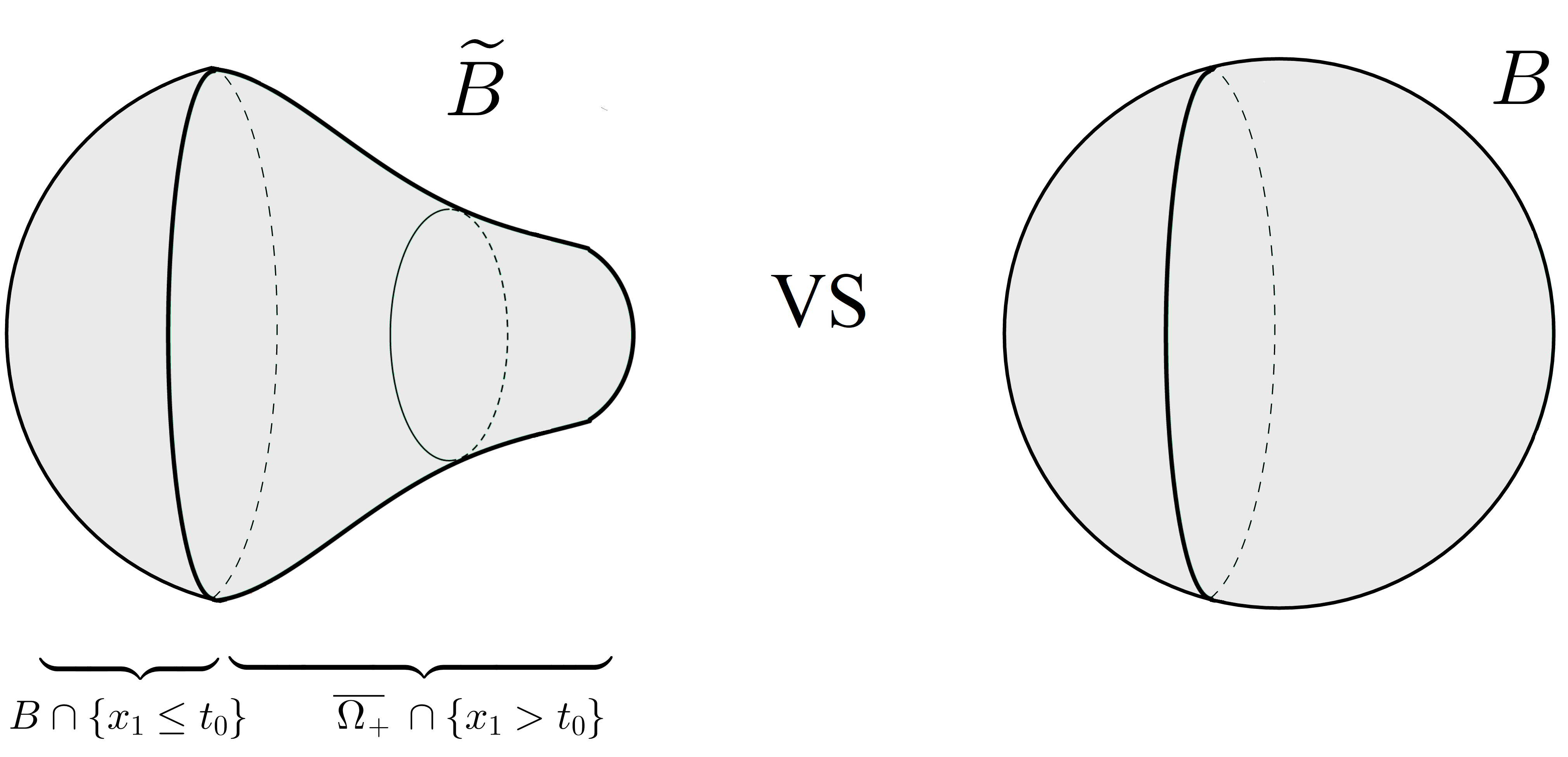}
\caption{The construction employed in the proof of Theorem \ref{thm spherical caps}.
}
\label{competitor}
\end{figure}

 Since $B$ and $\widetilde B$ have the same volume by construction, then, by the isoperimetric inequality, $\cH^{N-1}(\pa B)\le \cH^{N-1}(\pa \widetilde B)$, with equality holding true if and only if $B=\widetilde B$. By \eqref{prop of B} and the minimality of $\Om$, this implies that $\ol{\Om_+}\cap\{x_1>t_0\}=B\cap\{x_1>t_0\}$ and, hence, $\Ga_+\cap\{x_1>t_0\}$ must be a spherical cap. In particular, the set $I\cap [t_0,\infty)$ is connected. As a matter of fact, we claim that the set $I\cap(0,\infty)$ is connected. Indeed, let us assume, by contradiction, that there exists some value $t_1\in I \cap (0,t_0)$ that belongs to a different connected component. Now, performing one more time the rearrangement \eqref{a rearrangement} with $t_1$ instead of $t_0$ yields that the set $I\cap [t_1,\infty)$ must be connected, which is a contradiction. Therefore, $I\cap(0,\infty)$ is connected and $\Ga_+\cap \{x_1>t_0\}$ is a spherical cap. By analyticity, this implies that the whole $\Ga_+$ must be either a spherical cap or a sphere, as claimed. The same conclusion holds true for $\Ga_-$.

\end{proof}

%Theorem \ref{thm spherical caps} leaves us with just the following two possibilities for a minimizer.
Theorem \ref{thm spherical caps} suggests to us that the following two possibilities for a minimizer will be taken into account.

\begin{definition}\label{def type1 or type2}
Let $\Om$ be a minimizer of \eqref{min pb}. Then, $\Om$ is connected and $\Ga_\pm$ are spherical caps. This can only happen in one of the following two ways (see also Figure \ref{types}).
\begin{itemize}
\item {\bf Type~I minimizer.} The boundaries of the manifolds $\Ga_\pm$ coincide and $\overline{\Om}\cap\Sigma$ is an $(N-1)$-dimensional ball. In this case, $\cH^{N-1}(\Ga_0)=0$.
\item {\bf Type~II minimizer.} $\overline{\Om}\cap\Sigma$ is an $(N-1)$-dimensional ball, but this time the boundaries of the manifolds $\Ga_\pm$ are $(N-2)$-spheres whose radii have distinct lengths and centers do not necessarily coincide. In this case, $\cH^{N-1}(\Ga_0)>0$.
\end{itemize}

\end{definition}

\begin{figure}[ht]
\centering
\includegraphics[width=0.6\linewidth]{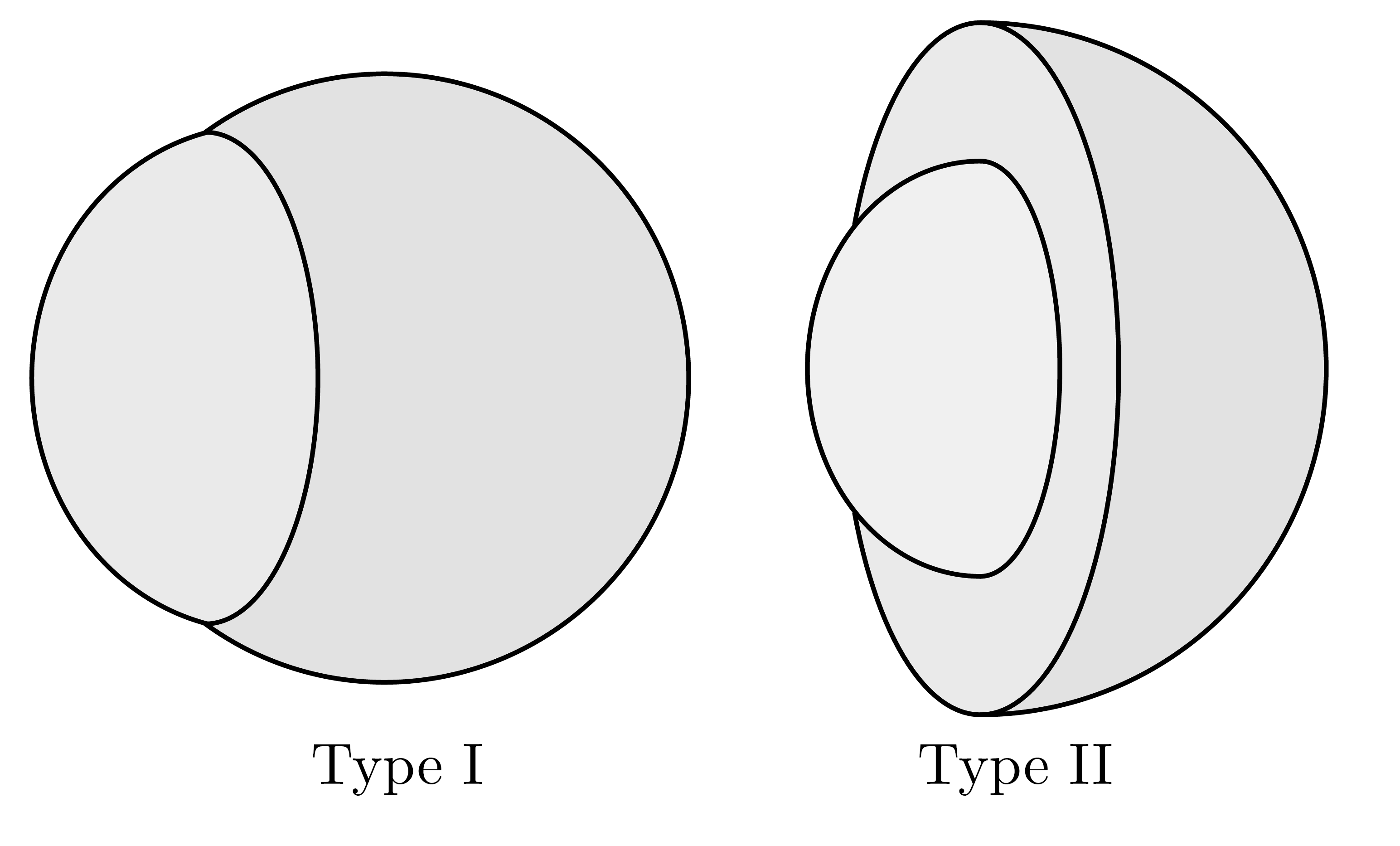}
\caption{The only two possible types of minimizers.
}
\label{types}
\end{figure}

%%%%%%%%%%%%%%%%%%%%%%%%%%%%%%%%%%%%%%%%%%%%%%%%%
%%%%%%%%%%%%%%%%%%%%%%%%%%%%%%%%%%%%%%%%%%%%%%%%%
%%%%%%%%%%%%%%%%%%%%%%%%%%%%%%%%%%%%%%%%%%%%%%%%%
%%%%%%%%%%%%%%%%%  Section 3 begins  %%%%%%%%%%%%%%%%%%%%%%%
%%%%%%%%%%%%%%%%%%%%%%%%%%%%%%%%%%%%%%%%%%%%%%%%%
%%%%%%%%%%%%%%%%%%%%%%%%%%%%%%%%%%%%%%%%%%%%%%%%%
%%%%%%%%%%%%%%%%%%%%%%%%%%%%%%%%%%%%%%%%%%%%%%%%%

\section{Existence of a minimizer and its regularity} \label{ch existence}

Here, we will finally prove the existence of a minimizer for Problem \eqref{min pb}. In order to do this, we will first need to generalize the rearrangement technique employed in the proof of Theorem \ref{thm spherical caps} to a general set of finite perimeter. Indeed, the situation is quite standard in each half-space since we deal then with a classical isoperimetric
problem, but the difficulty is to deal with the part of the boundary which may be on $\Sigma$ (since, a priori, we have little information about the regularity of the set there).

\begin{lemma}\label{Schwarz symmetrization B}
Let $\Om$ be a set of finite perimeter in $\rn$ with $|\Om|<+\infty$. Then, there exists a set of finite perimeter $\Omega^\sharp$ in $\mathbb R^N$ which is axially symmetric with respect to $\ell$ such that each of $\Ga^\sharp_\pm = \partial\Omega^\sharp \cap \rn_\pm$ is either a spherical cap or a sphere and the following hold:
 $$
|\Om^\sharp_\pm| = |\Om^*_\pm|,\  P(\Omega^\sharp, \rn_\pm)  \le P(\Omega^*, \rn_\pm),\ \cH^{N-1}(\Ga^\sharp_0) \le \cH^{N-1}(\Ga^*_0) \ \mbox{ and }\  \cF(\Om^\sharp) \le \cF(\Om^*),
$$
 respectively, where $\Om^\sharp_\pm$ and $\Ga^\sharp_0$ follow notations \eqref{some notations}. 
\end{lemma}
\begin{proof}
By Lemma \ref{Schwarz symmetrization A} we know that the set $\Om^*$ is well defined. Now, using the same notation as in Lemma \ref{Schwarz symmetrization A}, by \eqref{identity 2 for n star} and \eqref{estimates from above for every n} we have that, for every $n \in \mathbb N$
 $$
 \cH^{N-1}(\Omega^*\cap\partial H_{-\varepsilon_n})  \le P(\Omega^*)\ \mbox{ and similarly }  \cH^{N-1}(\Omega^*\cap\partial H_{\varepsilon_n})  \le P(\Omega^*).
 $$
 Therefore,  by the Bolzano-Weierstrass theorem,  up to a subsequence, we may assume that, as $n \to \infty$, the sequences of the radii of the  $(N-1)$-dimensional balls $\Omega^*\cap\partial H_{\pm\varepsilon_n}\ (n \in \mathbb N)$ centered at  $\ell\cap\partial H_{\pm\varepsilon_n}$ converge to some nonnegative numbers $r_\pm$, respectively. Let $D_\pm$ be the two $(N-1)$-dimensional closed balls in $\Sigma$ centered at $\ell\cap\Sigma$ with radii $r_\pm$, respectively. Let $B_\pm$ denote the two closed balls in $\mathbb R^N$ determined by
 \begin{equation}
 \label{construction of the two spherical caps}
 B_\pm \cap \Sigma = D_\pm\ \mbox{ and }\ |B_\pm \cap \mathbb R^N_\pm| = |\Omega^*_\pm|,
 \end{equation}
 respectively. 

\begin{figure}[ht]
\centering
\includegraphics[width=0.6\linewidth]{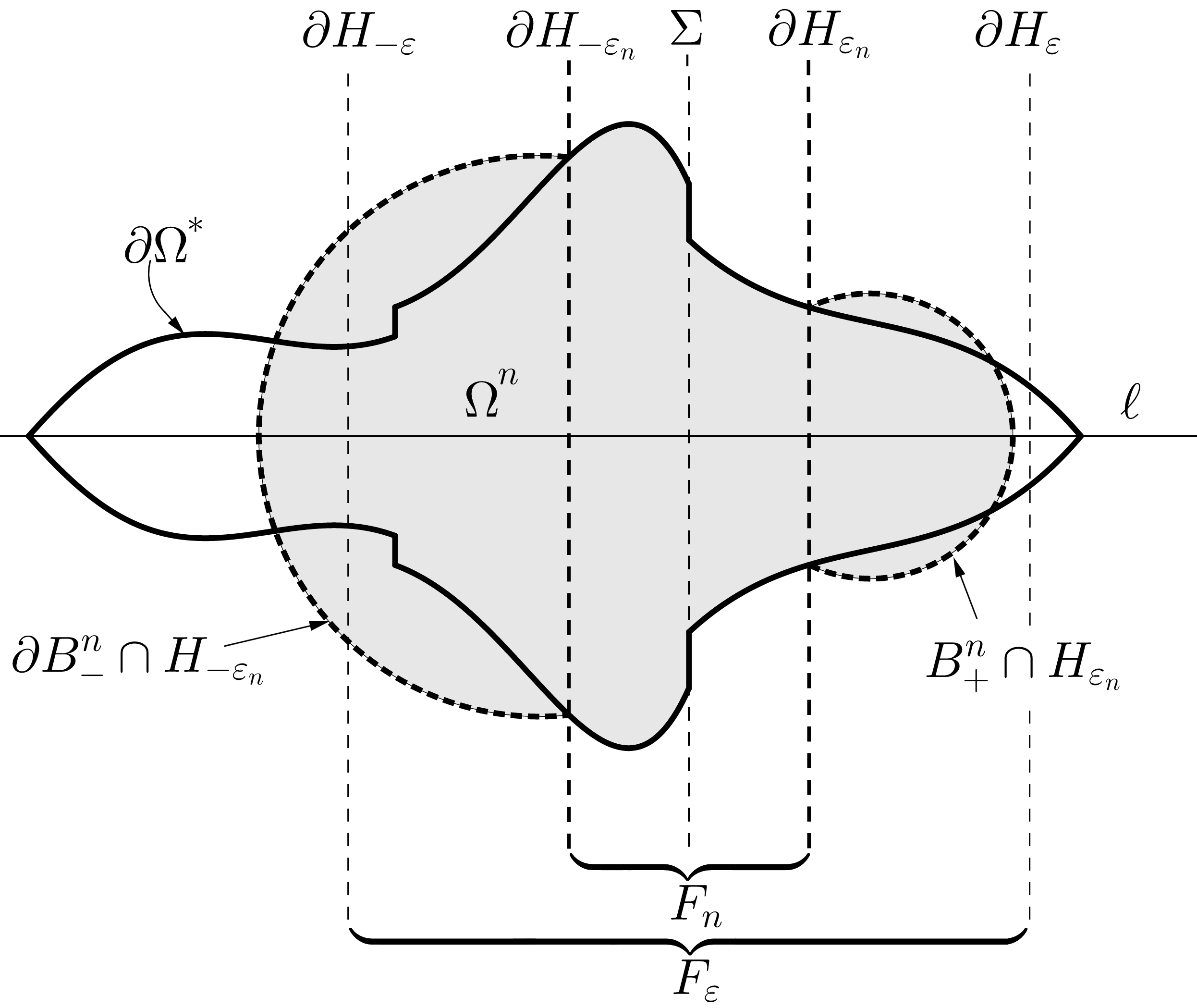}
\caption{The construction employed in the proof of Lemma \ref{Schwarz symmetrization B}. For ease of understanding, the boundary of the set $\Om^*$ is denoted by a bold line, while the interior of $\Om^n$ is shaded. 
} 
\label{construction}
\end{figure}

 Then we set
 \begin{equation}
 \label{construction of Omega sharp}
 \Omega^\sharp = (B_-\cap \overline{ \mathbb R^N_-}) \cup (B_+\cap\overline{ \mathbb R^N_+}).
 \end{equation}
 Hence $|\Omega^\sharp_\pm| = |\Omega^*_\pm|$, respectively. Notice that if $r_+ (\mbox{or } r_-)$ equals $0$, then $\overline{\Omega^\sharp_+} (\mbox{or } \overline{\Omega^\sharp_-} )$ equals the ball $B_+(\mbox{or } B_-)$. 
 Moreover, for every $n \in \mathbb N$, let $B^n_\pm$ denote the two closed balls in $\mathbb R^N$ determined by
 \begin{equation}
 \label{construction of the two approximate spherical caps}
 B^n_\pm \cap \partial H_{\pm\varepsilon_n} = \overline{\Omega^*}\cap\partial H_{\pm\varepsilon_n}\ \mbox{ and }\ |B^n_\pm \cap H_{\pm\varepsilon_n}| = |\Omega^* \cap H_{\pm\varepsilon_n}|,
 \end{equation}
 respectively.  Then, for every $n \in \mathbb N$, we set 
  \begin{equation}
 \label{construction of Omega sharp_n}
 \Omega^n = (B^n_-\cap \overline{H_{-\varepsilon_n}}) \cup (\Omega^*\cap F_n) \cup (B^n_+\cap \overline{ H_{\varepsilon_n}}).
 \end{equation}
 Hence $|\Omega^n_\pm| = |\Omega^*_\pm|$, respectively.
 Define 
\begin{equation*}
\widetilde B^n_\pm = \left( \Omega^*\cap H_{\pm\varepsilon_n}  \right) \cup (B^n_\pm \setminus H_{\pm\varepsilon_n}),
\end{equation*} 
respectively.  Since $B^n_\pm$ and $\widetilde B^n_\pm$ have the same volume respectively, we have from the classical isoperimetric inequality that $\cH^{N-1}(\pa B^n_\pm)\le P(\widetilde B^n_\pm)$ respectively. Hence it follows that for every $n \in \mathbb N$
$$
\cH^{N-1}(\pa B^n_\pm\cap \overline{H_{\pm\varepsilon_n}}) \le P(\Omega^*, H_{\pm\varepsilon_n}),
$$
respectively.
Then, we observe that for every $n \in \mathbb N$
$$
P(\Omega^n, \rn_\pm)  \le P(\Omega^*, \rn_\pm)\ \mbox{ and }\ \cH^{N-1}(\Ga^n_0) = \cH^{N-1}(\Ga^*_0), 
$$
 respectively, where $\Ga^n_0 = \partial\Omega^n \cap\Sigma$.  Another observation is that  $\{\Omega^n\}_n$ converges to  $\Omega^\sharp$ 
 in the sense of their characteristic functions as $n \to \infty$. Hence the lower semicontinuity  of the perimeter yields that 
 \begin{equation}
 \label{a use of the lower semicontinuity}
 P(\Om^\sharp, \rn_\pm) \le \liminf_{n \to \infty} P(\Om^n, \rn_\pm) \le P(\Omega^*, \rn_\pm) \mbox{ and }\ P(\Om^\sharp, F_\varepsilon) \le \liminf_{n \to \infty} P(\Om^n, F_\varepsilon),
 \end{equation}
 where $\varepsilon > 0 $ is an arbitrary number and $F_\varepsilon = \{ x \in \mathbb R^N : -\varepsilon < x_1 < \varepsilon \}$.
 Since  
 $$
 P(\Om^n, F_\varepsilon) = \cH^{N-1}(\pa\Om^n \cap (F_\varepsilon\setminus F_n)) + P(\Om^*, F_n), 
 $$
  we see that 
 \begin{equation*}
 \liminf_{n \to \infty} P(\Om^n, F_\varepsilon)   = \liminf_{n\to\infty} \cH^{N-1}(\pa\Om^n \cap (F_\varepsilon\setminus F_n)) +\cH^{N-1}(\Ga^*_0).
 \end{equation*}
By recalling that as $n \to \infty$ the sequence of radii of the  $(N-1)$-dimensional balls $\Omega^*\cap\partial H_{\pm\varepsilon_n}\ (n \in \mathbb N)$
converges to the nonnegative numbers $r_\pm$, respectively, we infer that
$$
\liminf_{\varepsilon \to 0} \liminf_{n\to\infty} \cH^{N-1}(\pa\Om^n \cap (F_\varepsilon\setminus F_n))  = 0,
$$
and hence
$$
\liminf_{\varepsilon \to 0}\liminf_{n \to \infty} P(\Om^n, F_\varepsilon) = \cH^{N-1}(\Ga^*_0).
$$
Moreover, 
$$
\lim_{\varepsilon \to 0} P(\Om^\sharp, F_\varepsilon) = |\cH^{N-1}(D_+)-\cH^{N-1}(D_-)| = \cH^{N-1}(\Ga^\sharp_0).
$$
Therefore it follows from \eqref{a use of the lower semicontinuity} that 
$$
\cH^{N-1}(\Ga^\sharp_0) \le  \cH^{N-1}(\Ga^*_0),\ P(\Om^\sharp, \rn_\pm) \le P(\Omega^*, \rn_\pm)\ \mbox{ and hence } \cF(\Om^\sharp) \le \cF(\Om^*),
$$
 which completes the proof.  
\end{proof}

\begin{theorem}\label{thm type1 or type2}
The problem \eqref{min pb} has a minimizer $\Omega$. Moreover it is of one of the two types described in Definition {\rm \ref{def type1 or type2}}. 
\end{theorem}
\begin{proof}
Let $\Omega^k$ be a minimizing sequence. Following \eqref{some notations} we introduce $\Omega^k_\pm$,
$\Gamma^k_\pm$ and $\Gamma^k_0$. 
%Let $H$ be the hyperplane $x_1=0$ and $D_1$ the line of direction $e_1$ (the first vector of the canonical basis passing through the origin). 
%We also introduce $\omega^n_\pm=\partial \Omega^n_\pm \cap H$ in such a way that $\Gamma^n_0= \omega^n_+ \Delta \omega^n_-$
%(the symmetric difference of these two sets).
%Since a spherical symmetrization around the line $D_1$ decreases the perimeter and preserves the volume,
%By applying the same rearrangement procedure described in the proof of Theorem \ref{thm spherical caps}, we can assume that each $\Omega^n_\pm$ is a spherical dome.
By Lemma \ref{Schwarz symmetrization B}, we may assume that  $\Ga^k_\pm$ are spherical caps.
%radially symmetric around this line and therefore 
%$\omega^n_\pm$ is a $N-1$-ball of radius $R^n_\pm$ respectively. 
Moreover, let $D^k_\pm$ denote the two ($N-1$)-dimensional balls given by the intersections $\ol\Om^k_\pm\cap\Sigma$ and let $R_\pm^k$ be their radii. In particular, since $\Om^k$ is a minimizing sequence for the functional $\cF$, we know that the sequence of radii $R_\pm^k$ must be bounded (as $\cH^{N-1}(\Ga_\pm^k)$ would diverge to infinity otherwise). This means that the whole sequence of sets $\Om^k$ is contained in a large enough compact set.
By compactness (compact embedding from
the BV-space to $L^1_{\rm loc}$, see \cite[Chapter 2]{HP}), up to extracting a subsequence, we can assume that:
\begin{itemize}
\item $\Omega^k_\pm$ converges to some $\Omega_\pm$ (in the sense of characteristic functions),
\item each $D^k_\pm$ converges to a ball $D_\pm$ (convergence of their radii).
\end{itemize}
Moreover, by property of this convergence, in particular the lower semicontinuity of the perimeter, we have
$$
\rho_\pm |\Om_\pm|=V_\pm, \quad P(\Omega, \rn_\pm)\leq \liminf_{k \to \infty} P(\Omega^k, \rn_\pm).
$$
Together with the convergence of the balls $D^k_\pm$ and the fact that $\cH^{N-1}(\Gamma_0)=|\cH^{N-1}(D_-)-\cH^{N-1}(D_+)|$, this implies that $\cF(\Omega)\leq \liminf_{k \to \infty} \cF(\Omega^k)$ and thus $\Omega$ is a minimizer of  problem \eqref{min pb}. 
%Concerning the regularity statement, it is a classical consequence of regularity of quasi-minimizers
%of the perimeter as proved in \cite{Amb-Pao}, \cite{Tam}.
The second part of the statement directly follows from Theorem \ref{thm spherical caps}.
In particular, to rule out the case of a sphere in one of the half-space, we can use
Corollary \ref{corsphere} below.
\end{proof}

%%%%%%%%%%%%%%%%%%%%%%%%%%%%%%%%%%%%%%%%%%%%%%%%%
%%%%%%%%%%%%%%%%%%%%%%%%%%%%%%%%%%%%%%%%%%%%%%%%%
%%%%%%%%%%%%%%%%%%%%%%%%%%%%%%%%%%%%%%%%%%%%%%%%%
%%%%%%%%%%%%%%%%%  Section 4 begins  %%%%%%%%%%%%%%%%%%%%%%%
%%%%%%%%%%%%%%%%%%%%%%%%%%%%%%%%%%%%%%%%%%%%%%%%%
%%%%%%%%%%%%%%%%%%%%%%%%%%%%%%%%%%%%%%%%%%%%%%%%%
%%%%%%%%%%%%%%%%%%%%%%%%%%%%%%%%%%%%%%%%%%%%%%%%%

\section{Complete characterization of the two types of minimizers} \label{ch snell law}

\subsection{Some preliminary geometrical lemmas}
By Theorem \ref{thm type1 or type2} we know that there are only two kinds of minimizers, which, without loss of generality, can be assumed to be symmetric with respect to rotations around the $x_1$ axis. Under this additional symmetry assumption, each candidate set becomes a (piecewise) hypersurface of revolution, whose generatrix can be uniquely described by four parameters (the radii $R_\pm$, and the angles of incidence $\al$ and $\beta$, or the radii $R_\pm$ and the pair of centers $a$ and $b$) as shown in Figure \ref{crosssection}. We will refer to each candidate set by means of those parameters as $\Om(\al,\be,R_\pm)$.
\begin{figure}[ht]
\centering
\includegraphics[width=0.7\linewidth]{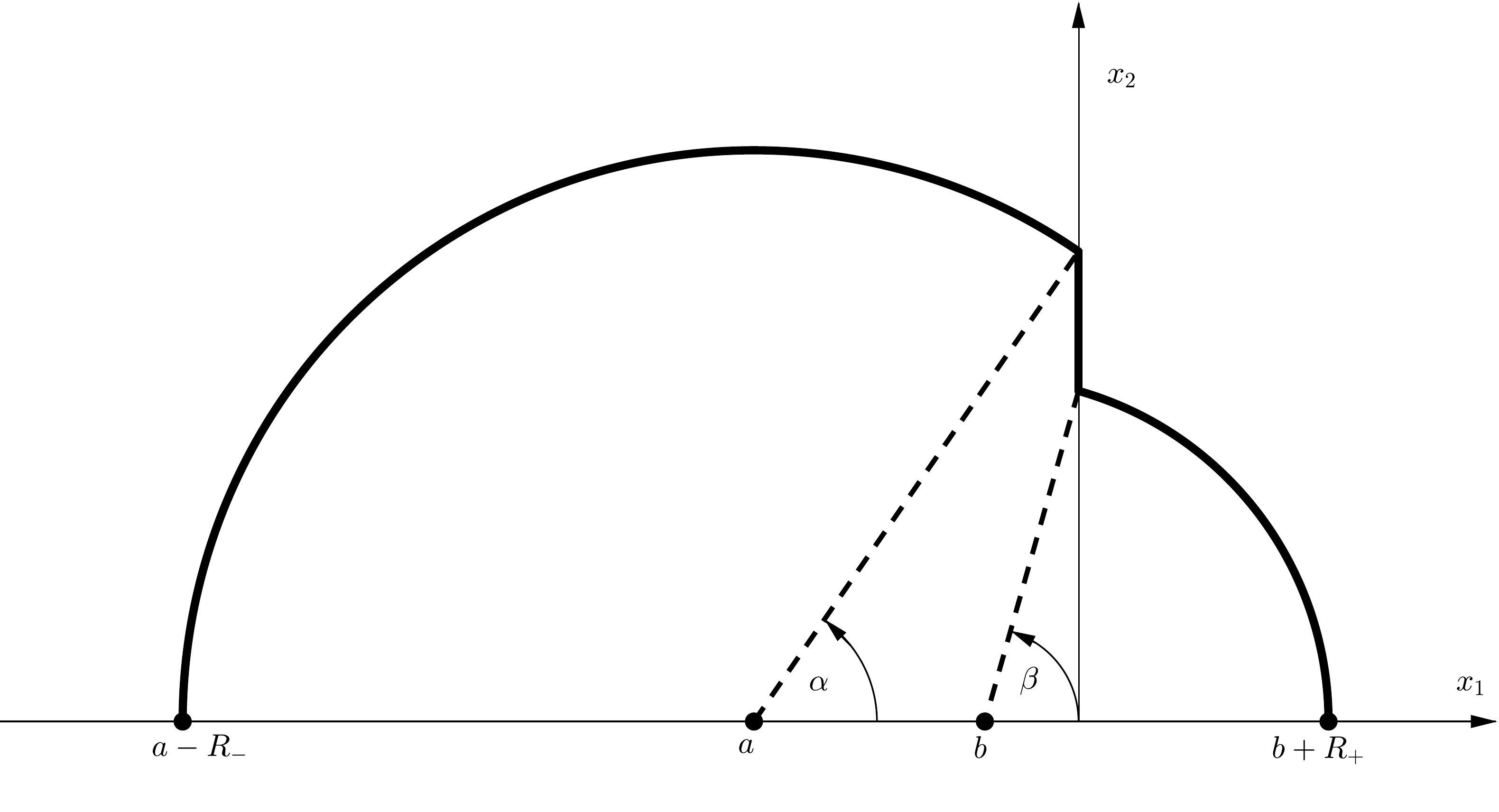}
\caption{Cross section of a minimizer (of type~II).
} 
\label{crosssection}
\end{figure}
In what follows, we are going to give explicit formulas for computing the Lebesgue measure of $\Om_\pm(\al,\be,R_\pm)$ and the ($N-1$)-dimensional Hausdorff measure of $\Ga_\pm(\al,\be,R_\pm)$ and $\Ga_0(\al,\be,R_\pm)$. We will also employ the use of the following shorthand notation:
\begin{equation}\label{J,I}
\begin{aligned}
&J_- =J_-(\al)= \int_\al^\pi (\sin \theta)^{N-2}\,d\theta, \quad 
&J_+ = J_+(\be)=\int_0^\be (\sin \theta)^{N-2}\,d\theta, \\
&I_- =I_-(\al)= \int_\al^\pi (\sin \theta)^{N}\,d\theta, \quad 
&I_+ = I_+(\be)=\int_0^\be (\sin \theta)^{N}\,d\theta, \\
& s_-=s_-(\al)=\sin \al, \quad s_+=s_+(\be)=\sin\be,\\%\quad
&c_-=c_-(\al)=\cos\al, \quad c_+=c_+(\be)=\cos\be.
\end{aligned}
\end{equation}

\begin{lemma}\label{some geometric formulas}
Let $\Om=\Om(\al,\be,R_\pm)$ and let $\om$ denote the {\rm($N-1$)}-dimensional volume of the {\rm($N-1$)}-dimensional unit ball. Then the following formulas hold true.
\begin{equation*}
    |\Om_\pm|=\om R_\pm ^N I_\pm,\quad \cH^{N-1}(\Ga_\pm) = (N-1)\om R_\pm^{N-1}J_\pm, \quad \cH^{N-1}(\Ga_0)= \om\, \left|(R_-s_-)^{N-1} - (R_+s_+)^{N-1}\right|.
\end{equation*}
\end{lemma}
\begin{proof}
We will prove only the formulas corresponding to the subscript $_-$, as the others are analogous. 
First, by Cavalieri's principle as 
\begin{equation*}
    |\Om_-|=\om\int_{-R_-}^{R_-\cos\al} \left(\sqrt{R_-^2-x^2}\right)^{N-1}\, dx. 
\end{equation*}
Now, the substitution $x=R_-\cos\te$ yields
\begin{equation*}
    |\Om_-|= \om \int_\al^\pi R_-^N (\sin\te)^N\, d\te = \om R_-^N I_-.
\end{equation*}
The value of $\cH^{N-1}(\Ga_\pm)$ can be easily computed by considering $\Gamma_-$ as a hypersurface of revolution in $\rn$.
Indeed, if we set $y(x)=\sqrt{R_-^2-x^2}$, by the area formula for surfaces of revolution in general dimension given in \cite{abe-agra} we get:
\begin{equation*}
    \cH^{N-1}(\Ga_-) = (N-1)\,\om \int_{-R_-}^{R_-\cos\al} y(x)^{N-2} \sqrt{1+(y'(x))^2}\, dx.
\end{equation*}
Now, recalling that $\sqrt{1+(y'(x))^2}= R_-/y(x)$ and performing the substitution $x=R_-\cos\te$ yield 
\begin{equation*}
    \cH^{N-1}(\Ga_-) = (N-1)\, \om \int_\al^\pi R_-^{N-1} (\sin\te)^{N-2}\, d\te = (N-1)\,\om R_-^{N-1} J_-.
\end{equation*}
Finally, the value of $\cH^{N-1}(\Ga_0)$ is immediately computed by noticing that, in this case, $\Ga_0$ is just an ($N-1$)-dimensional spherical shell whose outer and inner radii are given by $\max(R_- s_-, R_+s_+)$ and $\min(R_- s_-, R_+s_+)$ respectively.  
\end{proof}

\begin{corollary}\label{corsphere}
If $\Om$ is an optimal set, then neither of $\Ga_\pm$ can be a whole sphere.
\end{corollary}
\begin{proof}
Without loss of generality, assume that $\Om$ is a type~II minimizer and the part of boundary $\Gamma_-$ is a whole sphere.
This corresponds to $\alpha=0$ in the previous notations. To get the thesis, it suffices to prove
that the total perimeter strictly decreases when we increase $\alpha$. Let us first assume that
$\Gamma_+$ is not a whole sphere, then $\Gamma_0$ is not empty and $\cH^{N-1}(\Gamma_0)$
will decrease if $\alpha$ increases. Let us now look at $\cH^{N-1}(\Gamma_-)$. By the volume constraint and Lemma \ref{some geometric formulas}, the radius of any spherical
cap (including the case of the sphere where $\alpha=0$)  is given by
$$R_-(\alpha)=\left(\frac{V_-}{\omega\rho_-I_-(\alpha)}\right)^{1/N}$$
while the perimeter is given in terms of $\alpha$ by
$$\cH^{N-1}(\Gamma_-(\alpha))=(N-1)\omega J_-(\alpha) \left(\frac{V_-}{\omega\rho_-I_-(\alpha)}\right)^{(N-1)/N}.$$
Therefore, we want to study the dependence on $\alpha$ (near $\alpha=0$) of the function
$$g:\alpha \mapsto J_-(\alpha) \left(I_-(\alpha)\right)^{(1-N)/N}.$$
Its derivative is given by
$$g'(\alpha)=\sin^{N-2}(\alpha) I_-(\alpha)^{\frac{1}{N}-2}\left((1-\frac{1}{N})J_-(\alpha)\sin^2\alpha
- I_-(\alpha)\right)$$
that shows that the derivative is negative when $\alpha$ goes to 0 proving the claim.\\
When both $\Gamma_-$ and $\Gamma_+$ are whole spheres, the proof works as well
replacing each sphere by a spherical cap such that $\Gamma_0$ remains of zero ($N-1$)-measure (in other words, we can replace both $\Ga_\pm$ with two ``slightly perturbed" spherical caps in such a way that the two boundaries of the manifolds $\Gamma_\pm$ coincide).
\end{proof}

\begin{lemma}\label{NI=}
The following identity holds true for all $\al$ and $\beta$.
\begin{equation*}
    NI_\pm=(N-1)J_\pm \mp s_\pm^{N-1}c_\pm.
\end{equation*}
\end{lemma}
\begin{proof}
The identity corresponding to the subscript $_-$ follows from integration by parts:
\begin{equation*}%\label{I+=something}
I_-=    \int_\al^\pi (\sin\te)^N\,d\te = (\sin\al)^{N-1}\cos\al + (N-1) \int_\al^\pi (\sin\te)^{N-2}(\cos\te)^2\,d\te. 
\end{equation*}
Now, since $(\cos\te)^2 = 1-(\sin\te)^2$ for all $\te$, we get 
\begin{equation*}
    I_- = s_-^{N-1}c_- + (N-1)J_- -(N-1)I_-,
\end{equation*}
which is exactly what we wanted. The case corresponding to the subscript $+$ is analogous and will be therefore omitted.
\end{proof}

\subsection{Characterization of type~I minimizers}

\begin{theorem}[Snell's law for type~I minimizers]
Let $\Om=\Om(\al,\be,R_\pm)$ be a type~{\rm I} minimizer for \eqref{min pb}. Then the following identity holds true.
\begin{equation*}
    \rho_-\cos\al=\rho_+\cos\be.
\end{equation*}
\end{theorem}

Let $\Om$ be a type~I minimizer of \eqref{min pb} and let $\Om_\pm$ and $\Ga_\pm$ be defined according to \eqref{some notations}. Let $U=\Om\cap\{x_2=0\}$, $U_\pm= U\cap \rn_\pm$, and $U_0=U\cap \{x_1=0\}$. Now, the set $\pa\Om\cap\{x_2>0\}$ can be expressed as the graph of a function $u:U\to (0,\infty)$. In particular, since we know that $\pa\Om$ is the union of two spherical caps, we get 
\begin{equation}\label{u=}
    u(x) = \begin{cases}
    \sqrt{R_-^2-(x-a e_1)^2} \quad \text{for }x\in U_-,\\
    \sqrt{R_+^2-(x-b e_1)^2} \quad \text{for }x\in U_+,\\
    \end{cases}
\end{equation}
where $e_1=(1,0,\dots, 0)\in \rn$, $a e_1$ and $b e_1$ are the centers of the spheres that generate $\Ga_\pm$ respectively, and $R_\pm$ are the corresponding radii. Notice that, since $\Om$ is of type~I by hypothesis, the function $u$ defined in \eqref{u=} admits a continuous extension along the interface $U_0$. 

Let $\rho=\rho_- \cX_{U_-}+\rho_+ \cX_{U_+}$, where $\cX_A$ denotes the indicator function of the set $A$ (namely, $\cX_A(x)=1$ if $x\in A$ and $\cX_A(x)=0$ otherwise). 
For any function $w\in H_0^1(U)$, let
\begin{equation*}
    \cG(w)=\rho\int_{U} \sqrt{1+|\gr w|^2} \,dx %+ \int_{D_+} \rho_+ \sqrt{1+|\gr w|^2}\, dx
    .
\end{equation*}
Notice that, by construction, %we have 
%\begin{equation*}
 $   \cF(\Om)=2\cG(u)$.  
%\end{equation*}
In particular, since by definition, $\Om$ is a minimizer of \eqref{min pb}, then $u$ must be a critical point for the following Lagrangian
\begin{equation*}
    \cL(w) = \cG(w)+ \mu_-\int_{U_-} w \, dx + \mu_+\int_{U_+} w \, dx, 
\end{equation*}
where $\mu_\pm$ are the Lagrange multipliers associated to the volume constraints on $\Om_\pm$ respectively.
Therefore, for all $v\in H_0^1(U)$, we must have
\begin{equation*}
\restr{\frac{d}{dt}}{t=0} \cL(u+tv)=0.    
\end{equation*}
An explicit computation of the G\^ateaux derivative above yields 
\begin{equation}\label{weak form}
   \int_U \rho \frac{\gr u \cdot \gr v}{\sqrt{1+|\gr u|^2}}\, dx + \int_U \mu v \, dx =0 \quad \text{for all } v\in H_0^1(U), 
\end{equation}
where we set $\mu=\mu_-\cX_{U_-}+\mu_+ \cX_{U_+}$. 
Equation \eqref{weak form} is nothing else than the weak form of 
\begin{equation*}
    -\dv\left( \rho\frac{\gr u}{\sqrt{1+|\gr u|^2}}\right)=\mu \quad\text{in }U.
\end{equation*}
By a standard result concerning elliptic PDE's in divergence form with piecewise constant coefficients, we get that the quantity $$\rho\frac{\gr u\cdot e_1}{\sqrt{1+|\gr u|^2}}$$ has no jump along the interface $U_0$. 
An explicit computation with \eqref{u=} at hand yields

%\begin{equation*}
%    \rho_-\, \dfrac{\dfrac{(x-ae_1)\cdot e_1}{u(x)}}{\sqrt{1+\dfrac{(x-ae_1)^2}{u^2(x)}}}\, =\, 
 %   \rho_+\, \dfrac{\dfrac{(x-be_1)\cdot e_1}{u(x)}}{\sqrt{1+\dfrac{(x-be_1)^2}{u^2(x)}}}\quad \text{for }x\in D_0.
%\end{equation*}

\begin{equation}\label{no jump}
    \rho_-\, \dfrac{\dfrac{(x-ae_1)\cdot e_1}{\sqrt{R_-^2-(x-a e_1)^2}}}{\sqrt{1+\dfrac{(x-ae_1)^2}{{R_-^2-(x-a e_1)^2}}}}\, =\, 
    \rho_+\, \dfrac{\dfrac{(x-be_1)\cdot e_1}{\sqrt{R_+^2-(x-b e_1)^2}}}{\sqrt{1+\dfrac{(x-be_1)^2}{{R_+^2-(x-b e_1)^2}}}}\quad \text{for }x\in U_0.
\end{equation}
Since, by construction, $\sqrt{R_-^2-(x-a e_1)^2}=
    \sqrt{R_+^2-(x-b e_1)^2} $ for $x\in U_0$, the equality in \eqref{no jump} simplifies to 
    \begin{equation*}
        \dfrac{\rho_- a}{R_-} =     \dfrac{\rho_+ b}{R_+}.
    \end{equation*}
Or, equivalently
\begin{equation}\label{type1 snell}
    \rho_- \cos\al=\rho_+\cos\be.
\end{equation}
\subsection{Characterization of type~II minimizers}
\begin{theorem}[Snell's law for type~II minimizers]\label{thm snell II}
Let $\Om=\Om(\al,\be,R_\pm)$ be a type~{\rm II} minimizer for \eqref{min pb}. If  $R_-\sin\al > R_+\sin\be$, then 
\begin{equation}\label{Omega_ga}
    \rho_-\cos\al= \ga= \rho_+\cos\be.
\end{equation}
On the other hand, if $R_-\sin\al < R_+\sin\be$, then 
\begin{equation}\label{Omega_-ga}
    \rho_-\cos\al= -\ga= \rho_+\cos\be.
\end{equation}
%In particular $\al, \be \in (0,\pi/2)$
\end{theorem}
\begin{proof}
For $\al\in[0,\pi)$ and $\be\in(0,\pi]$, let us define the following 
\begin{equation}\label{R_- and R_+}
   R_-:= R_-(\al)=\sqrt[\leftroot{-3}\uproot{5}N]{\frac{V_-}{\rho_- \om I_-(\al)}}, \quad
R_+:=R_+(\be)=\sqrt[\leftroot{-3}\uproot{5}N]{\frac{V_+}{\rho_+ \om I_+(\be)}}.
\end{equation}
This way, by the first formula of Lemma \ref{some geometric formulas}  we are sure that the volume constraints 
\begin{equation}\label{volume constraint}
\rho_\pm \left|\Om_\pm (\al, \be, R_-,R_+)\right|=V_\pm
\end{equation}
are satisfied. 
Notice that, differentiating \eqref{R_- and R_+} yields
\begin{equation}\label{R'}
\frac{d}{d\al} R_-(\al) = \frac{s_-^N R_-}{NI_-}, \quad %\text{and} \quad
\frac{d}{d\be} R_+(\be) = -\frac{s_+^N R_+}{NI_+}. 
\end{equation}

Let now $\Om(\al,\be)=\Om(\al,\be, R_-(\al), R_+(\be))$ be a type~II minimizer of \eqref{min pb}.% which satisfies $R_-(\al)\sin\al>R_-(\be)\sin\be$.
In particular, the pair $(\al,\be)$ is a minimizer of the functional $\cF(\al,\be)=\cF\left(\Om(\al,\be) \right)$: 
\begin{equation}\label{F(al,be)}
\begin{aligned}
    \cF(\al,\be)= \rho_-\cH^{N-1}(\Ga_-(\al))+ \rho_+\cH^{N-1}(\Ga_+(\be))+ \ga\om \left| R_-^{N-1}(\al)s_-^{N-1} -   R_+^{N-1}(\be)s_+^{N-1}\right|\\
    = 
    (N-1)\rho_- \om R_-^{N-1}J_-+ (N-1)\rho_+ \om R_+^{N-1}J_+ + \ga \om \left| R_-^{N-1}s_-^{N-1} - R_+^{N-1}s_+^{N-1}   \right|.
\end{aligned}
\end{equation}
Moreover, since $(\al,\be)$ is a minimizer of $\cF(\cdot,\cdot)$ by hypothesis, we have
\begin{equation*}
    \frac{d}{d\al}\cF(\al,\be)= 0 = \frac{d}{d\be}\cF(\al,\be). 
\end{equation*}
In what follows, we will compute the partial derivative of $\cF$ with respect to the first variable, $\al$, at  $(\al,\be)$. %Partial derivatives with respect to $\al$ will be simply denoted by $'$.  
Differentiating \eqref{F(al,be)} with respect to $\al$ with \eqref{R'} at hand, under the assumption that $R_-(\al)\sin\al >R_+(\be)\sin\be$ yields
\begin{equation*}
    0=\frac{d}{d\al}\cF(\al,\be)= (N-1)\om R_-^{N-1} \left( \rho_- (N-1) \frac{s_-^N J_-}{NI_-} - \rho_-s_-^{N-2}+ \ga  \frac{s_-^{2N-1}}{NI_-} + \ga s_-^{N-2}c_- \right).
\end{equation*}
Now, by Lemma \eqref{NI=}, we get
\begin{equation}\label{d/d al F(al,be)}
  0=\frac{d}{d\al}\cF(\al,\be)= \underbrace{\frac{(N-1)\om R_-^{N-1} s_-^{N-2}}{NI_-} \left( (N-1) J_- c_- + s_-^{N-1}\right)}_{>0} (\ga-\rho_-c_-).
\end{equation}
This implies that 
\begin{equation*}
\rho_-\cos\al=\ga    
\end{equation*}
as wanted. The condition concerning the angle $\be$ is analogous. 
Finally, the optimality condition \eqref{Omega_-ga} follows from \eqref{Omega_ga} by replacing $\ga$ by $-\ga$ in \eqref{F(al,be)}. 
\end{proof}

\begin{lemma}\label{L1 L2}
Let 
\begin{equation*}
    L_1(\al)=\frac{\sin^N(\al)}{I_-(\al)} \quad \text{and}\quad L_2(\be)= \frac{\sin^N(\be)}{I_+(\be)}.
\end{equation*}
Then $L_1$ is a strictly increasing function in the interval $(0,\pi)$, while $L_2$ is strictly decreasing in the same interval.
\end{lemma}
\begin{proof}
We will just show that the function $L_1$ is strictly increasing, since the proof for $L_2$ is analogous. 
First of all, we compute the derivative of $L_1$ with respect to $\al$:
\begin{equation*}
    \frac{d}{d\al} L_1(\al)= {\frac{s_-^{N-1}}{I_-^2} }\left(  N c_- I_- + s_-^{N+1}\right).
\end{equation*}
We will show that $ f_1=N c_- I_- + s_-^{N+1}$ is strictly positive for all $\al\in[0,\pi)$. In particular, notice that $f_1(\pi)=0$, hence it suffices to show hat  $f_1$ is strictly decreasing in the interval $(0,\pi)$. 
Another derivative with respect to $\al$ yields
\begin{equation*}
\frac{d}{d\al}f_1(\al)= s_- \left( -N I_- + s_-^{N-1} c_- \right)    
\end{equation*}
We claim that $f_2= -N I_- + s_-^{N-1} c_-$ is negative in the interval $(0,\pi)$.   
To this end, notice that $f_2(\pi)=0$, hence we just need to show that $f_2$ is a strictly increasing function in the interval $(0,\pi)$. 
This is indeed true, as
\begin{equation*}
    \frac{d}{d\al} f_2= (N-1)s_-^{N}+(N-1)s_-^{N-2}c_-^2= (N-1)s_-^N>0.
\end{equation*}
Therefore $L_1$ is a strictly increasing function as claimed. 
\end{proof}

\begin{lemma}\label{three candidates}
Let $\rho_\pm>0$, $0\le\ga<\min\{\rho_-,\rho_+\}$ and $V_\pm>0$ be given. Then the following three candidate minimizers $\Om_\ga$, $\Om_{-\ga}$ and $\Om^\star$ are well defined (and uniquely characterized by their defining properties).
\begin{itemize}
    \item[\rm (i) ] The set $\Om_\ga$ is the unique candidate minimizer of type~{\rm II} of the form $\Om(\al,\be,R_\pm)$ that satisfies both the volume constraints \eqref{volume constraint} and the Snell's law \eqref{Omega_ga}. 
    \item[\rm (ii) ] The set $\Om_{-\ga}$ is the unique candidate minimizer of type~{\rm II} of the form $\Om(\al,\be,R_\pm)$ that satisfies both the volume constraints \eqref{volume constraint} and the Snell's law \eqref{Omega_-ga}.
    \item[\rm (iii) ] The set $\Om^\star$ is the unique candidate minimizer of type~{\rm I} of the form $\Om(\al,\be, R_\pm)$ that satisfies the volume constraints \eqref{volume constraint}, the Snell's law \eqref{type1 snell} and the equality $R_-\sin\al=R_+\sin\be$. 
\end{itemize}
\begin{comment}
\textcolor{red}{write well definedness for $\Omega_\ga, \Om_{-\ga}$ and $\Om^\star$}

there exist only one quadruple $(\al^\star,\be^\star,R_-^\star,R_+^\star)\in (0,\pi)^2\times (0,\infty)^2$ such that the set $\Om^\star=\Om(\al^\star,\be^\star,R_\pm^\star)$ satisfies the following conditions.
\begin{itemize}
    \item $\rho_\pm|\Om_\pm^\star|=V_\pm$,
    \item $\rho_- \cos\al^\star=\ga=\rho_+\cos\be^\star$.
\end{itemize}
\end{comment}
\end{lemma}
\begin{proof}
The points (i)-(ii) can be treated together, by considering the set $\Om_\ga$ that satisfies conditions \eqref{volume constraint}--\eqref{Omega_ga} for $|\ga|<\min(\rho_-,\rho_+)$. Notice that this amounts to solving a nonlinear system of 4 equations in 4 variables, which nicely decouples. We get
\begin{equation}\label{decoupled solution}
    \begin{aligned}
\al&=\arccos\left(\frac{\ga}{\rho_-}\right), \quad &\be=\arccos\left(\frac{\ga}{\rho_+}\right),\\
   R_-&=\sqrt[\leftroot{-3}\uproot{5}N]{\frac{V_-}{\rho_- \om I_-(\al)}}, \quad
&R_+=\sqrt[\leftroot{-3}\uproot{5}N]{\frac{V_+}{\rho_+ \om I_+(\be)}}.
    \end{aligned}
\end{equation}

A key observation to show  (iii) relies on the fact that any set $\Om^\star$ satisfying \eqref{volume constraint}--\eqref{type1 snell} is indeed a particular case of $\Om_\ga$ satisfying \eqref{volume constraint}--\eqref{Omega_ga} for some constant $\ga\in\RR$ to be determined. We just need to determine the value of $\ga$ such that $(\al,\be, R_\pm)$, defined by \eqref{decoupled solution}, satisfy $R_-\sin\al=R_+\sin\be$. To this end we will look for the zeros of the following function
\begin{equation*}
L(\ga)=R_-^N \sin^N(\al) - R_+^N\sin^N(\be) = \frac{V_-}{\om\rho_-} L_1(\al) - \frac{V_+}{\om \rho_+}L_2(\be),
\end{equation*}
    where $\al$ and $\be$ are considered to be functions of $\ga$ by the first two relations in \eqref{decoupled solution} and the functions $L_1$, $L_2$ are defined in Lemma \eqref{L1 L2}. Now, since $\al$ and $\be$ are strictly decreasing functions of $\ga$, Lemma \eqref{L1 L2} implies that $\ga\mapsto L(\ga)$ is strictly decreasing too. This ensures uniqueness for $\Om^\star$. As far as existence is concerned, one could analyze the limit cases and conclude by the intermediate value theorem or simply notice that the existence of a type~I candidate minimizer derives from the existence of a minimizer of \eqref{min pb} for $\ga>\min(\rho_-,\rho_+)$ (i.e. when no candidate minimizers of type~II can be defined).   
%In showing $(iii)$, we will prove existence and uniqueness separately. First, notice that neither $\Om_\ga$ nor $\Om_{-\ga}$ can be defined when $\ga>\min(\rho_-,\rho_+)$. Nevertheless, by Theorem \textcolor{red}{[cite existence theorem]}, the existence of a minimizer for \eqref{min pb} is guaranteed even for $\ga>\min(\rho_-,\rho_+)$. This implies the existence of a candidate minimizer of type~I, $\Om^*$. A key observation in order to show that only one such $\Om^*$ exists is the following: the set $\Om^*$ satisfies conditions \eqref{volume constraint}--\eqref{Omega_ga} for some $|\ga|<\min()$      
\end{proof}

In what follows, let $(\al^\star,\be^\star)\in (0,\pi)^2$ denote the unique pair such that
\begin{equation*}
    \Om^\star=\Om(\al^\star,\be^\star).
\end{equation*}
Furthermore, let 
\begin{equation}\label{ga^*}
    \ga^\star=\rho_-\cos\al^\star \; (=\rho_+\cos\be^\star).
\end{equation}
Notice that the sign of $\ga^\star$ is determined by the given constants $V_\pm$ and $\rho_\pm$. Indeed if $V_-/\rho_- \gtreqless V_+/\rho_+$, then $\ga^\star\gtreqless 0$. This is a consequence of the fact that 
\begin{equation}\label{L(0)}
    L(0)= \frac{1}{\om \int_{\pi/2}^\pi \sin^N(\te)\,d\te} \left( \frac{V_-}{\rho_-}-\frac{V_+}{\rho_+}\right)  
\end{equation}
and the map $\ga\mapsto L(\ga)$ is strictly decreasing, as stated in the proof of point (iii) of Lemma \eqref{three candidates}.  

%%%%%%%%%%%%%%%%%%%%%%%%%%%%%%%%%%%%%%%%%%%%%%%%%
%%%%%%%%%%%%%%%%%%%%%%%%%%%%%%%%%%%%%%%%%%%%%%%%%
%%%%%%%%%%%%%%%%%%%%%%%%%%%%%%%%%%%%%%%%%%%%%%%%%
%%%%%%%%%%%%%%%%%  Section 5 begins  %%%%%%%%%%%%%%%%%%%%%%%
%%%%%%%%%%%%%%%%%%%%%%%%%%%%%%%%%%%%%%%%%%%%%%%%%
%%%%%%%%%%%%%%%%%%%%%%%%%%%%%%%%%%%%%%%%%%%%%%%%%
%%%%%%%%%%%%%%%%%%%%%%%%%%%%%%%%%%%%%%%%%%%%%%%%%

\section{Proof of the main result} \label{ch proof main thm}

We are ready to prove the main result of this paper.
\begin{theorem} 
Let $\ga\ge0$, $\rho_\pm>0$ and $V_\pm>0$ be given. Moreover, without loss of generality let $V_-/\rho_- > V_+/\rho_+$. Then, the minimizers of \eqref{min pb} can be characterized as follows.
\begin{itemize}
\item[\rm (i) ] If $\ga<\ga^\star$, $\Om_\ga$ is the only minimizer of \eqref{min pb} up to suitable translations. 
\item[\rm (ii) ]  If $\ga=\ga^\star$, then $\Om_{\ga^\star}=\Om^\star$ is the only minimizer of \eqref{min pb} up to suitable translations.  
\item[\rm (iii) ]  If $\ga>\ga^\star$, then $\Om^\star$ is the only minimizer of \eqref{min pb} up to suitable translations.
\end{itemize}
\end{theorem}
\begin{proof}
By Lemma \ref{three candidates} we know that (up to suitable translations) there are only three candidates minimizers for \eqref{min pb}, namely $\Om_\ga$, $\Om_{-\ga}$ and $\Om^\star$. First of all, we will show that the set of candidate minimizers can be reduced to just $\Om_{\ga}$ and $\Om^\star$. Indeed, if $\ga=0$, then $\Om_\ga=\Om_{-\ga}$. The case $\ga>0$ is a bit more complicated. First, notice that $L(0)>0$ by \eqref{L(0)}. Moreover, since the map $\ga\mapsto L(\ga)$ is strictly decreasing, then $L(-\ga)>0$: this implies that $R_-\sin\al>R_+\sin\be$ and hence, by Theorem \ref{thm snell II}, $\Om_{-\ga}$ should satisfy \eqref{Omega_ga} (instead of \eqref{Omega_-ga}) in order to be a minimizer. 

We will now prove part (i) of the theorem by contradiction. Let $\ga<\ga^\star$ and assume by contradiction that the candidate of type~I, $\Om^\star=\Om(\al^\star,\be^\star)$, is a minimizer. In particular, this implies that $\al^\star$ is a minimum point of the functional $f(\al)=\cF\left(\Om(\al,\be^\star)\right)$ which was previously explicitly computed in \eqref{F(al,be)}. Since $\Om(\al^\star,\be^\star)$ is of type~I by construction, the term inside the absolute value bars in \eqref{F(al,be)} vanishes at $(\al^\star,\be^\star)$. In particular, by the first part of Lemma \ref{L1 L2}, we know that the term inside the absolute value bars in \eqref{F(al,be)} is negative for $\be=\be^\star$ and $\al<\al^\star$, and positive for $\be=\be^\star$ and $\al>\al^\star$. Therefore, the same computations that lead to \eqref{d/d al F(al,be)} give us  
\begin{equation*}%\label{f'}
\frac{d}{d\al}    f(\al) = \begin{cases}
    A(\al) (-\ga-\rho_-\cos\al)\quad \text{for }0<\al<\al^\star,\\
    A(\al) (\ga-\rho_-\cos\al)\quad \text{for }\al^\star<\al<\pi,
        \end{cases}
\end{equation*}
where 
\begin{equation*}
A(\al)=\frac{(N-1)\om R_-^{N-1} s_-^{N-2}}{NI_-} \left( (N-1) J_- c_- + s_-^{N-1}\right)>0.   
\end{equation*}
In particular, since 
\begin{equation*}
    \ga<\ga^\star=\rho_-\cos\al^\star
\end{equation*}
by assumption, both left and right derivatives of $f$ at $\al^\star$ are negative, which violates the assumption made about $\al^\star$ being a minimum point of $f$.

Part  (ii) is obvious because, by definition, the characteristic property of $\ga^\star$ is 
\begin{equation*}
    \Om_{\ga^\star}=\Om^\star.
\end{equation*}
Since we previously ruled out $\Om_{-\ga}$ as a competitor, we obtain that $\Om_{\ga^\star}=\Om^\star$ is the unique minimizer up to suitable translations.

We will now take $\ga>\ga^\star$ and prove part  (iii) of the theorem. It will suffice to show that $\Om_\ga$ is not a minimizer. Indeed, since $L(\ga^\star)=0$ by construction, and $\ga\mapsto L(\ga)$ is strictly decreasing by Lemma \ref{L1 L2}, we get that $L(\ga)<0$. In other words, $\Om_\ga$ satisfies $R_-\sin\al<R_+\sin\be$. If $\Om_\ga$ were indeed a minimizer, then by Theorem \ref{thm snell II} it should satisfy the Snell's law \eqref{Omega_-ga} (instead of \eqref{Omega_ga}). Since $\ga$ is not $0$ in this case, this is a contradiction.  
\end{proof}

\begin{small}

\end{small}
\end{document}